\documentclass[11p]{amsart}
\usepackage{mathtools}
\usepackage{amssymb}
\usepackage{amsfonts}
\usepackage{amsthm}
\usepackage{mathrsfs}
\usepackage{polynom,comment}
\usepackage{mathbbol}
\usepackage{latexsym,amscd,soul}
\usepackage[all]{xy}
\usepackage{color}

\usepackage{tikz}
\usepackage{epic}

\newtheorem{thm}{Theorem}
\newtheorem{prop}{Proposition}
\newtheorem{lem}{Lemma}
\numberwithin{equation}{section}
\numberwithin{prop}{section}
\numberwithin{lem}{section}
\numberwithin{thm}{section}
\newtheorem{cor}{Corollary}
\numberwithin{cor}{section}
\newtheorem{conj}{Conjecture}

\theoremstyle{definition}

\numberwithin{defn}{section}

\newtheorem{rem}{Remark}
\numberwithin{rem}{section}

\def \<{\langle}
\def \>{\rangle}

\def \a{\alpha }

\def \b{\beta }

\newcommand{\bea}{\begin{eqnarray}}
\newcommand{\eea}{\end{eqnarray}}
\newcommand{\be}{\begin {equation}}
\newcommand{\ee}{\end{equation}}
\newcommand{\g}{\mathfrak{g}}

\newcommand{\h}{\mathfrak{h}}

\newcommand{\C}{\Bbb C}

\newcommand{\Z}{\Bbb Z}

\newcommand{\xa}[2]{x^{\hat{\nu}}_{\alpha_{#1}}\left(#2\right)}

\newcommand{\xb}[1]{x^{\hat{\nu}}_{\alpha_1+\a_2}\left(#1\right)}

\newcommand{\hnu}{\hat{\nu}}
\newcommand{\n}{\mathfrak{n}}

\newcommand{\pt}{\psi_{\gamma,\theta}\tau_{\gamma,\theta}}

\begin{document}
\title[Principal Subspaces of Higher Level Standard $A_{2}^{(2)}$-Modules]{Presentations of Principal Subspaces of Higher Level Standard $A_2^{(2)}$-Modules}  
\author{Corina Calinescu, Michael Penn and  Christopher Sadowski}


\begin{abstract} 
We study the principal subspaces of higher level standard $A_2^{(2)}$-modules, extending earlier work in the level one case, by Calinescu$^{  \*}$, Lepowsky and Milas.  We prove natural presentations of principal subspaces and also of certain  related spaces. By using these presentations we obtain exact sequences, which yield recursions satisfied by the characters of the principal subspaces and related spaces. We conjecture a formula for a specialized character of the principal subspace, given by the Nahm sum of the inverse of the tadpole Cartan matrix.
\end{abstract}

\thanks{Corresponding author: Corina Calinescu, {\em E--mail address}: ccalinescu@citytech.cuny.edu}
\thanks{C.C was partially supported by the Simons Foundation Collaboration Grant for Mathematicians, and by PSC-CUNY Research Awards.}

\maketitle

\section{Introduction}
 The principal subspaces of standard (integrable highest weight) modules introduced in \cite{FS1}-\cite{FS2} have been studied by several authors from different standpoints. Our approach is based on vertex operator algebra theory (\cite{B}, \cite{FLM2}, \cite{FHL}, \cite{LL}). Algebraic and combinatorial properties, such as presentations, combinatorial bases and graded dimensions, of the principal subspaces of certain modules for untwisted affine Lie algebras were proved in \cite{CLM1}-\cite{CLM2}, \cite{CalLM1}-\cite{CalLM3}, \cite{Bu1}-\cite{Bu3} and other works. Analogous results in the case of twisted affine Lie algebras appear in in \cite{CalLM4}, \cite{CMP},  \cite{PS1}-\cite{PS2}, \cite{BS}, and in \cite{MP}, \cite{P}, \cite{PSW} for lattice vertex operator algebras. There are also ``commutative" principal subspaces studied in  \cite{Pr}, \cite{Je}, \cite{T1}-\cite{T2}, etc., and the quantum case was studied in \cite{Ko}.
 
In this paper, a continuation of \cite{CalLM4}, we switch our attention to level $k$ standard modules for $A_2^{(2)}$ and their principal subspaces, where $k$ is an integer with $k \geq 1$.  In this case we introduce new spaces, which we call {\it virtual subspaces}, since they are convenient for proving the results regarding the principal subspace. 
Denote by $\mathfrak{n}$ the Lie subalgebra of $A_2$ spanned by the root vectors for the positive roots, and by $\overline{\mathfrak{n}}[{\hnu}]$ an appropriate affinization of $\mathfrak{n}$. The virtual subspaces, denoted by $W_{k,i}$ are defined as $W_{k,i} = U(\bar{\frak{n}}[\hat{\nu}])\cdot v_{k,i}$, where $v_{k,i}$ are certain highest weight vectors.  With our definitions, we have that $W_{k, 0}$ is the principal subspace of the level $k$ standard $A_{2}^{(2)}$-module. The virtual subspaces are analogous to the principal-like subspaces in the untwisted case introduced in \cite{CalLM3}.   These subspaces were  implicitly used in the proof of all other twisted results, \cite{CalLM4}, \cite{CMP}, \cite{PS1}--\cite{PS2}, and \cite{PSW}, but the increase in the complexity of the current setting requires that these spaces to be worked with explicitly.

The virtual subspaces play a role similar to that of the principal subspaces of level $k$  non-vacuum modules for $A_2^{(1)}$, denoted by $W((k-i)\Lambda_0+i\Lambda_1)$ in \cite{CLM2} and \cite{CalLM2}, where $\Lambda_0, \Lambda_1$ are fundamental weights.  It was proved in \cite{CalLM2} that the principal subspaces $W((k-i)\Lambda_0+i\Lambda_1)$ have a presentation given by 
 an ideal generated by a single family of degree $k+1$ terms
 with an additional generator of $x_{\alpha}(-1)^{k-i+1}$, where $\alpha$ is the positive simple root. In this work we prove that our virtual subpaces  satisfy
 $$W_{k,i}\cong U(\overline{\mathfrak{n}}[{\hnu}])/I_{k,i}$$
 where the ideal $I_{k,i}$ is generated by $k+2$ different families of degree $k+1$ terms and a family of degree $k-i+1$ monomials constructed by using $\xa{1}{-\frac{1}{4}}$ and $\xb{-1}$, where $\alpha_1, \alpha_2$ are positive simple roots. These presentations of $W_{k,i}$ are given in Theorem \ref{presentation}, the main result of our paper.  The proof of this theorem adapts strategies from \cite{S1} and other results involving principal subspaces into a nested inductive argument. The outer induction descends through the virtual subspaces ending at the principal subspace $W_{k,0}$, while the inner induction homogenizes the extra terms in the ideals, $\xa{1}{-\frac{1}{4}}^m\xb{-1}^n$, by iteratively decreasing the power of $\xa{1}{-\frac{1}{4}}$, until the appropriate extra term is a power of $\xb{-1}$. In our proof, we make use of certain operators we call $\mathcal{Y}_i$, which play the role of the constant terms of intertwining operators used in \cite{CalLM1}-\cite{CalLM3}.
 
The paper is organized as follows. We recall the vertex operator constructions of $A_{2}^{(2)}$ and higher level standard modules for $A_{2}^{(2)}$ in sections 2 and 3. These results are standard and mostly taken from \cite{L1}, \cite{FLM1} (se also \cite{CalLM4}). Section 4 gives certain maps  that play an important role in proving the presentations of principal subspaces. In Section 5 we prove the main result of this paper, Theorem \ref{presentation} In the last section of the paper we construct short exact sequences of maps among virtual subspaces, and obtain, as a consequence, a set of recursions satisfied by their characters. Although this is not a complete system to allow us to solve for characters, we are able to conjecture that a specialized character of $W_{k,0}$ is given in terms of the inverse of the tadpole Cartan matrix of rank k. When $k$ is even, this character is related to Gollnitz-Gordon-Andrews identities.
 
{\bf Acknowledgements:} We thank the referee for providing us with constructive comments.

\section{Vertex operator construction of $A_2^{(2)}$}
In this section we recall from \cite{CalLM4} (which follows \cite{L1} and \cite{FLM1}, \cite{FLM2} and \cite{L2})  the vertex operator construction of $A_2^{(2)}$ using the lattice vertex algebra construction. 
Let  $\mathfrak{g}=sl(3, \mathbb{C})$  and $\h$ be a Cartan subalgebra of $\g$. Fix a root system $\Delta\subset\h^{*}$ and take $\{\a_1,\a_2\}$ to be a choice of simple roots. Identify $\h^{*}$ with $\h$ via a suitable symmetric, invariant bilinear form $\left<\cdot,\cdot\right>:\g\times \g\to \mathbb{C}$ such that 
\[
\begin{pmatrix}\left<\a_i,\a_j\right>\end{pmatrix}_{ij}=\begin{pmatrix}2&-1\\-1&2\end{pmatrix}.
\]
Now consider the root lattice of $\g$
\be
L=\mathbb{Z}\Delta=\mathbb{Z}\a_1\oplus\mathbb{Z}\a_2\subset \h,\ee
equipped with the bilinear form $\left<\cdot,\cdot\right>$. Take $\nu\in\text{Aut }L$ to be the isometry of $L$ given by the folding of the Dynkin diagram 

\begin{center}
  \begin{tikzpicture}[scale=.4]
  
 \draw [thick] (.35,0) -- (2,0);

  \draw [thick] (2.35,0) circle [radius=0.35];
      \draw [thick] (0,0) circle [radius=0.35];
    \node (a1) at (0,0) [below left ] {$\a_1$};
    \node (a1) at (2.35,0) [below right ] {$\a_2$};

  \end{tikzpicture}
\end{center}
in other words, we have 
\be
\nu(r\a_1+s\a_2)=r\a_2+s\a_1\ee
for any integers r and s.
Following the construction of twisted modules for lattice vertex algebras, we let $l$ be a positive integer such that $\nu^l=\text{id}_L$ and 
\be
\left<\nu^{l/2}\a,\a\right> \text{ for } \a\in L.\ee
In our setting this amounts to setting $l=4$, even though $\nu^2=\text{id}_L$. This doubling of the period was seen in \cite{CalLM4}, \cite{CMP}, and handled more generally in \cite{PSW}. Now that we have set the period of the isometry to be 4 we fix a primitive fourth root of unity, $\eta$, which we take $\eta=i$. 

We consider two central extensions of the $L$ by the cyclic group $\left<i\right>\cong\mathbb{Z}/4\mathbb{Z}$ which we denote by $\hat{L}$ and $\hat{L}_{\nu}$ with associated commutator maps 
\begin{align}
C_0:L\times L&\to \mathbb{C}^{\times}\\
(\alpha,\beta)&\mapsto (-1)^{\left<\alpha,\beta\right>}\end{align}
and 
\begin{align}
C:L\times L&\to \mathbb{C}^{\times}\\
(\alpha,\beta)&\mapsto -(-1)^{\left<\nu\alpha,\beta\right>},\end{align}
respectively. We also have normalized cocyles $\epsilon_{C_0}$ and $\epsilon_C$ and normalized sections ($\alpha\mapsto e_{\alpha}$) associated to these central extensions so that 
\be
e_{\alpha}e_{\beta}=\epsilon_{C_0}(\alpha,\beta)e_{\alpha+\beta} \text{ in } \hat{L}.\ee
For concreteness we can take 
\be 
\epsilon_{C_0}(m\alpha_1+n\alpha_2,r\alpha_1+s\alpha_2)=(-1)^{nr},\ee
where $m, n, r, s \in \Z$.

We also recall the affine Lie algebras
\begin{equation}
\hat{\h} = \h \otimes \C[t,t^{-1}] \oplus \C\bf{k}
\end{equation}
and
\begin{equation}
\hat{\h}[\nu] = \coprod_{n \in \frac{1}{k}\Z}\h_{(kn)} \otimes t^n \oplus \C\bf{k}
\end{equation}
with their usual brackets. We refer the reader to \cite{CalLM4} for further details regarding these constructions. 

Denote by $(V_L, Y)$ the vertex operator algebra associated with the root lattice $L$. We have, in particular,
\[
Y(\iota(e_{\alpha}), x)=E^{-} (-\alpha, x)E^{+}(-\alpha, x) e_{\alpha} x^{\alpha},
\]
where $\iota:\hat{L}\to V_L$ is the obvious inclusion.
Define the operators $x_{\alpha}(n)$ such that
\[
Y(\iota(e_{\alpha}), x)=\sum_{n \in \mathbb{Z}} x_{\alpha}(n) x^{-n-1}.
\]
Extend $\nu$ to an automorphism of $V_L$ and call it $\hat{\nu}$. Then $\hat{\nu} \in \mbox{ Aut} \; (V_L)$ and  $\hat{\nu}^4=1$. We have 
\be\hnu\iota(e_{\a_1})=i\iota(e_{\a_2}),\hnu\iota(e_{\a_2})=i\iota(e_{\a_1}), \text{ and } \hnu\iota(e_{\a_1+\a_2})=\iota(e_{\a_1+\a_2}).\ee

Denote by $(V_L^T, Y^{\hat{\nu}})$ the irreducible $\hat{\nu}$-twisted module for $V_L$ on which $\hat{\h}[\nu]$ has a natural action. We have 
\[
Y^{\hat{\nu}}(\iota(e_{\alpha}), x) =4^{-\frac{\langle \alpha, \alpha \rangle}{2}} \sigma(\alpha) E^{-}(-\alpha, x)E^{+}(-\alpha, x)e_{\alpha}x^{\alpha_{(0)}+\frac{\langle \alpha_{(0)}, \alpha_{(0)} \rangle}{2}-\frac{\langle \alpha, \alpha \rangle}{2}},
\]
where 
\be \label{normalizing_factor}
\sigma (\alpha)=(1+i)^{\langle \nu \alpha, \alpha \rangle}2^{\langle \alpha, \alpha \rangle /2}
\ee
is a normalizing factor. For $n \in (1/4) \mathbb{Z}$ define the operators $x_{\alpha}^{\hat{\nu}}(n)$ such that
\[
Y^{\hat{\nu}}(\iota(e_{\alpha}), x)=\sum_{n \in (1/4)\mathbb{Z}} x_{\alpha}^{\hat{\nu}}(n) x^{-n-\frac{\langle \alpha, \alpha \rangle}{2}}.
\]

Now we observe that the Lie algebra $\g$ may be realized as the vector space
\be
\g=\h\oplus\coprod_{\a\in\Delta}\mathbb{C}x_{\a}\ee
with $\h=\mathbb{C}\alpha_1\oplus\mathbb{C}\alpha_2$ and 
\[
[h,x_{\alpha}]=\left<h,\a\right>=-[x_{\a},h], [\h,\h]=0\]
\[[x_{\a},x_{\b}]=\begin{cases}\epsilon_{C_0}(\alpha,-\a)\a &\text{ if } \a+\beta=0\\
\epsilon_{C_0}(\a,\beta)x_{\a+\beta} &\text{ if } \left<\a,\beta\right>=-1\\
0 &\text{ if }\left<\a,\beta\right>\geq0\end{cases}\]
With our choice of cocycle the brackets of interest are
\[[x_{\a_j},x_{-\a_j}]=\a_j \text{ and } [x_{\a_1},x_{\a_2}]=x_{\a_1+\a_2}\]
for $j\in\{1,2\}$.

We now lift the isometry $\nu:L\to L$, which may also be viewed as an automorphism of the Lie subalgebra $\h\subset\g$ to an automorphism of $\g$ which we denote by $\hnu$. Explicitly we have
\be
\hnu x_{\a_1}=ix_{\a_2}, \hnu x_{\a_2}=ix_{\a_1}, \text{ and }\hnu x_{\a_1+\a_2}=x_{\a_1+\a_2}\ee
with similar formulas for elements associated to the negative roots. 

Define
\begin{equation}
\g_{(m)} = \{ x \in \g | \hnu x = i^m x \}
\end{equation}
for $m \in \mathbb{Z}$ and form 
\be
\g[\hnu]=\coprod_{n\in\frac{1}{4}\Z}\g_{(4n)}\otimes t^n\oplus \C\mathbf{c},\ee
the $\hnu$-twisted affine Lie algebra associated to $\g$ and $\hnu$ with
\be
[x\otimes t^m,y\otimes t^n]=[x,y]t^{m+n}+\left<x,y\right>m\delta_{m+n,0}\mathbf{c}\ee
and 
\be
[\mathbf{c},\g[\hnu]]=0\ee
for $m,n\in\frac{1}{4}\Z$, $x\in\g_{(4m)}$, and $y\in\g_{(4n)}$.  Adjoining the degree operator to $\g[\hnu]$ gives a copy of the twisted affine Lie algebra $A_2^{(2)}$ (cf. \cite{K}).

We recall the following result:
\begin{thm}[\cite{L1} Theorem 9.1 and \cite{FLM1} Theorem 3.1] The representation of $\h[\hnu]$ on $V_L^T$ extends uniquely to a Lie algebra representation of $\g[\hnu]$ on $V_L^T$ such that 
\[
\left(x_{\a}\right)_{(4n)}\otimes t^n\mapsto x_{\a}^{\hnu}(n)\]
for all $n\in\frac{1}{4}\Z$. Moreover, $V_L^T$ is an irreducible $\g[\hnu]$-module.\end{thm}

The following structure results from \cite{CalLM4} will be useful in this paper:

\begin{lem} [\cite{CalLM4}, Lemma 3.2]\label{lemma1} As operators on $V_L^T$, we have
\[
x_{\a_2}^{\hnu}(m)=x_{\a_1}^{\hnu}(m) \; \; \; \mbox{if} \; \; \; m \in \frac{1}{4} +\mathbb{Z}
\]
and 
\[
x_{\a_2}^{\hnu}(m)=-x_{\a_1}^{\hnu}(m) \; \; \; \mbox{if} \; \; \; m \in \frac{3}{4} +\mathbb{Z}.
\]
\end{lem}

\begin{lem} [\cite{CalLM4}, Lemma 3.3]\label{lemma2} As operators on $V_L^T$, we have 
\be\label{commutator1} [x_{\a_1}^{\hnu}(m),x_{\a_1}^{\hnu}(n)]=-\frac{i}{4}(i^{-4m}-(-i)^{-4m})x_{\a_1+\a_2}^{\hnu}(m+n) \; \; \; \mbox{if} \; \; \; m,n\in\frac{1}{4}+\frac{1}{2}\mathbb{Z}
\ee
and 
\be\label{commutator2} [x_{\a_1+\a_2}^{\hnu}(m),x_{\a}^{\hnu}(n)]=0 \; \; \; \mbox{if} \; \; \; m\in\mathbb{Z}, n\in\frac{1}{4}\mathbb{Z} \; \; \; \mbox{and} \; \; \;  \a\in\{\a_1,\a_2,\a_1+\a_2\}. 
\ee
\end{lem}

Consider the $\hat{\nu}$-stable Lie subalgebra of $\mathfrak{g}$: $\mathfrak{n}=\mathbb{C}x_{\alpha_1} \oplus \mathbb{C}x_{\alpha_2} \oplus \mathbb{C} x_{\alpha_1+\alpha_2}$, and its twisted affinization
\begin{equation} \label{n_affine}
\bar{\mathfrak{n}}[\hat{\nu}]=\coprod_{r \in (1/4)\mathbb{Z}} \mathfrak{n}_{(4r)} \otimes t^r.
\end{equation}

We recall the definition of the {\em principal subspace} of $V_L^T$ :
\[
W_L^T=U(\bar{\mathfrak{n}}[\hat{\nu}]) \cdot v_{\Lambda} \subset V_L^T.
\]
We shall denote $v_{\Lambda} \in V_L^T$ by 1: $v_{\Lambda}=1 \in V_L^T$.

We recall the following result, from which we will derive the necessary presentation of our principal subspace:
\begin{thm}[\cite{CalLM4}, Theorem 5.3] On the standard $\hat{\g}[\hnu]$-module $V_L^T$ we have:
\be\label{vertexoplevel1relation1}
\lim_{x_1^{1/4}\to x_2^{1/4}}(x_1^{1/2}+x_2^{1/2})Y^{\hnu}(\iota(e_{\a_j}),x_2)Y^{\hnu}(\iota(e_{\a_j}),x_1)=0 \text{ for } j=1,2,\ee
\be
\lim_{x_1^{1/4}\to ix_2^{1/4}}(x_2^{1/2}-x_1^{1/2})Y^{\hnu}(\iota(e_{\a_1}),x_2)Y^{\hnu}(\iota(e_{\a_2}),x_1)=0,\ee
\be\label{vertexoplevel1relation3}
Y^{\hnu}(\iota(e_{\a_1+\a_2}),x)^2=0,\ee
and
\be\label{vertexoplevel1relation4}
Y^{\hnu}(\iota(e_{\a_j}),x)Y^{\hnu}(\iota(e_{\a_1+\a_2}),x)=0 \text{ for } j=1,2.\ee
\end{thm}
Certain truncations of the coefficients of the products of vertex operators in (\ref{vertexoplevel1relation1}), (\ref{vertexoplevel1relation3}) and (\ref{vertexoplevel1relation4}), together with the highest weight vector relations, are all the relations of a presentation of the principal subspace $W_L^T$ (cf. Theorem 7.1 in \cite{CalLM4}).

Recall (5.18) from \cite{CalLM4} with $j=1$ and $t=1/2$:
\[
R^0_{1, \frac{1}{2}} =x_{\a_1}^{\hnu} \left ( -\frac{3}{4} \right ) x_{\a_1}^{\hnu} \left ( -\frac{1}{4} \right ) + x_{\a_1}^{\hnu} \left ( -\frac{1}{4} \right ) x_{\a_1}^{\hnu} \left ( -\frac{3}{4} \right ) + 2x_{\a_1}^{\hnu} \left ( -\frac{1}{4} \right )^2,
\]
which implies 
\be \label{relation_A22}
\left (x_{\a_1}^{\hnu} \left ( -\frac{3}{4} \right )x_{\a_1}^{\hnu} \left ( -\frac{1}{4} \right ) + x_{\a_1}^{\hnu} \left ( -\frac{1}{4} \right ) x_{\a_1}^{\hnu} \left ( -\frac{3}{4} \right ) \right  ) \cdot 1=0.
\ee
 As a consequence of Lemmas \ref{lemma1} and \ref{lemma2}, and (\ref{relation_A22}) we note that
\be \label{commutator_dot_1}
x_{\a_1+\a_2}^{\hnu}(-1) \cdot 1= c x_{\a_1}^{\hnu} \left ( -\frac{3}{4} \right ) x_{\a_1}^{\hnu} \left ( -\frac{1}{4} \right ) \cdot 1,
\ee
where $c$ is a nonzero constant.

Finally, as in \cite{CalLM4}, \cite{CMP}, etc. we have a tensor product grading on $V_L^T$ given by the eigenvalues of $L^{\hnu}(0)$, where \begin{equation*} Y^{\hnu}(\omega,z)=\sum_{m\in\mathbb{Z}}L^{\hnu}(m)z^{-m -2},\end{equation*}
 which we call the {\em weight grading}.  We also have a grading by {\em charge}, given by the eigenvalues of $\gamma = (\alpha_1 + \alpha_2)_{(0)}$. We note that these gradings are compatible, and refer the reader to \cite{CalLM4} for more details.

\section{Vertex operator construction of higher level standard modules for $A_2^{(2)}$} 

Let $k$ be  a nonnegative integer. We will use $k$ to denote the level of representations in this paper. Consider the vector space
\[
V_L^{\otimes k}= V_L \otimes \cdots \otimes V_L
\]
and
the vertex operator 
\[
Y^{\otimes k} (v_1 \otimes \cdots \otimes v_k, x)=Y(v_1, x) \otimes \cdots \otimes Y(v_k, x).
\]

Define the vectors
\be 
v_{k,0} = 1 \otimes \dots \otimes 1,
\ee
and, more generally, for $0 \leq i \leq k$,
\be \label{hw_vectors}
v_{k,i} = 1^{\otimes (k-i)} \otimes e_{\alpha_1}^{\otimes i} \ (= \underbrace{1 \otimes \cdots \otimes 1}_{k-i \; \; \mbox{times}} \otimes \underbrace{e_{\alpha_1} \otimes \cdots \otimes e_{\alpha_1}}_{i \; \mbox{times}}).
\ee

It is known that $(V_L^{\otimes k}, Y^{\otimes k})$ is a vertex operator algebra (cf. \cite{LL}). Define the operators $x_{\alpha}(m)$ on $V_L^{\otimes k}$ such that 
\[
Y^{\otimes k}(x_{\alpha}(-1) \cdot v_{k,0}, x)=\sum_{m \in \mathbb{Z}}x_{\alpha}(m)x^{-m-1}.
\]

Denote by $(V_L^T)^{\otimes k}$ the tensor product of k copies of $V_L^T$:
\[
(V_L^T)^{\otimes k}=V_L^T \otimes \cdots \otimes V_L^T.
\]
Consider the automorphism $\hat{\nu} \otimes \cdots \otimes \hat{\nu}$ of $V_L^{\otimes k}$, and denoted it by $\hat{\nu}$. Then $\hat{\nu}^4=1$.  Let $Y^{\hat{\nu}, \otimes k}$ be the vertex operator
\[
Y^{\hat{\nu}, \otimes k}=Y^{\hat{\nu}} \otimes \cdots \otimes Y^{\hat{\nu}}.
\]
Since $V_L^T$ is an irreducible $\hat{\nu}$-twisted $V_L$-module, then $((V_L^T)^{\otimes k}, Y^{\hat{\nu}, \otimes k})$ is an irreducible $\hat{\nu}$-twisted module for $V_L^{\otimes k}$. Define the operators $x_{\alpha}^{\hat{\nu}}(m)$ on $(V_L^T)^{\otimes k}$ as follows:
\[
Y^{\hat{\nu}, \otimes k}(x_{\alpha}(-1) \cdot v_{k, 0}, x)= \sum_{m \in \frac{1}{4} \mathbb{Z}}x_{\alpha}^{\hat{\nu}}(m)x^{-m-1}.
\]

By \cite{Li} we know that $L^{\hnu}(k\Lambda_0) :=U(\hat{\mathfrak{g}}[\hat{\nu}]) \cdot v_{k,0} \subset (V_L^T)^{\otimes k}$ is a level k standard $\hat{\mathfrak{g}}[\hat{\nu}]$-module with the tensor product action (diagonal action)
\begin{equation} \label{diagonal_action}
x_{\alpha}^{\hat{\nu}}(n) \cdot v_{k, 0}=x_{\alpha}^{\hat{\nu}}(n) \cdot 1 \otimes 1 \cdots \otimes 1 + \cdots + 1 \otimes \cdots 1 \otimes x_{\alpha}^{\hat{\nu}}(n) \cdot 1.
\end{equation}
We also note here that our gradings on $V_L^T$ by weight and charge naturally extend to $(V_L^T)^{\otimes k}$ (cf. \cite{BS} for more details).

Define the vector spaces
\begin{equation} \label{virtual_spaces}
W_{k,i} = U(\bar{\frak{n}}[\hat{\nu}])\cdot v_{k,i}
\end{equation}
for any $0 \leq i \leq k$ (recall (\ref{n_affine}) and (\ref{hw_vectors})), which we call {\it virtual subspaces}. Note that 
\be
W_{k,0}=U(\bar{\mathfrak{n}}[\hat{\nu}]) \cdot v_{k,0}
\ee
 is the principal subspace of the level k standard module $L^{\hnu}(k\Lambda_0)$. These virtual subspaces are analogous to the principal-like subspaces introduced in \cite{CalLM3}.

\begin{thm}\label{RelThm} On the $\hat{\mathfrak{g}}[\hat{\nu}]$-module $L^{\hnu}(k\Lambda_0)$ we have 
\be \label{vertexoprel1}\begin{aligned} 
\left ( \prod_{1 \leq i<j \leq k+1-r} \lim_{x_i^{1/4} \to x_j^{1/4}} (x_i^{1/2}\right.&\left.+x_j^{1/2}) \right)\\& Y^{\hat{\nu},\otimes k}(\iota(e_{\alpha_l}), x_1) \cdots  Y^{\hat{\nu},\otimes k}(\iota(e_{\alpha_l}), x_{k+1-r} ) \left(Y^{\hat{\nu},\otimes k}(\iota(e_{\alpha_1+\alpha_2}), x)\right)^r=0 \end{aligned}
\ee
for $0 \leq r \leq k-1$ and 
\be \label{vertexoprel3} 
Y^{\hat{\nu},\otimes k}(\iota(e_{\alpha_l}), x)Y^{\hat{\nu},\otimes k}(\iota(e_{\alpha_1+\alpha_2}), x )^k =0
\ee
for $l=1,2$. We also have 
\be\label{vertexoprel2}
Y^{\hnu, \otimes k}(\iota(e_{\a_1+\a_2}),x)^{k+1}=0.\ee
\end{thm}

\begin{proof}
Recall that
\begin{equation} \begin{aligned} \nonumber
& Y^{\hat{\nu}, \otimes k} \left ( \iota(e_{\alpha}),  x \right)  \\   & = Y^{\hat{\nu}}(\iota(e_{\alpha}), x) \otimes 1_V \otimes \cdots \otimes 1_V+ \cdots + 1_V \otimes \cdots \otimes Y^{\hat{\nu}}(\iota(e_{\alpha}), x),
\end{aligned}
\end{equation}
where $\alpha$ is any root and $1_V$ is the identity operator. Now the statement follows from the corresponding statement for the level 1 case (Theorem 5.3 in \cite{CalLM4}).
\end{proof}


Following \cite{CalLM1}--\cite{CalLM4}, \cite{CMP}, etc. and using Theorem \ref{RelThm}, we introduce the following formal infinite sums indexed by $t \in \frac{1}{4} \mathbb{Z}$:
\be  \label{R-relations}  \begin{aligned} 
 & R(\alpha_1, \alpha_1+\alpha_2, r | t) &
\\ & = \sum_{\substack{ i_1, \dots, i_{k+1-r}  \in \{ 0, \dots \frac{k-r}{2} \} \\ i_1+ \cdots i_{k+1-r} = \frac{(k-r)(k-r+1)}{4} \\  m_1, \dots m_{k+1-r} \in \frac{1}{4} +\frac{1}{2} \mathbb{Z}, m_{k+2-r}, \dots , m_{k+1} \in \mathbb{Z} \\ m_1+\cdots + m_{k+1}+  \frac{(k-r)(k-r+1)}{4}=-t }} \left (  \prod_{j=1}^{k+1-r} x_{\alpha_1}^{\hat{\nu}}(m_j+i_j) \right ) x_{\alpha_1+\alpha_2}^{\hat{\nu}}(m_{k+2-r}) \cdots x_{\alpha_1+\alpha_2}^{\hat{\nu}}(m_{k+1}),&
\end{aligned}
\ee

\be
R(\alpha_1, \alpha_1+\alpha_2, k | t)=  \sum_{\substack{m_1 \in \frac{1}{4} +\frac{1}{2} \mathbb{Z}, m_2, \dots, m_{k+1} \in \mathbb{Z} \\ m_1+ \cdots +m_{k+1}=-t}} x_{\alpha_1}^{\hat{\nu}}(m_1) x_{\alpha_1+\alpha_2}^{\hat{\nu}} (m_2) \cdots x_{\alpha_1+\alpha_2}^{\hat{\nu}}(m_{k+1})
\ee
and
\be
R(\alpha_1+\alpha_2 | t)= \sum_{\substack{m_1, \dots, m_{k+1} \in \mathbb{Z} \\m_1+ \cdots +m_{k+1}=-t} } x_{\alpha_1+\alpha_2}^{\hat{\nu}}(m_1) \cdots x_{\alpha_1+\alpha_2}^{\hat{\nu}}(m_{k+1}).
\ee

As in \cite{CalLM1}--\cite{CalLM4}, \cite{CMP}, \cite{PS1}--\cite{PS2}, etc. we may write 
\be \label{R-0-relations}
R(\a_1,\a_1+\a_2,r|t)=R^0(\a_1,\a_1+\a_2,r|t) + a \ee
where $a\in \widetilde{U(\overline{\n}[\hnu])\overline{\n}[\hnu]_{+}}$ and $R^0(\a_1, \a_1+\a_2, r|t)$ is the finite sum associated to $R(\a_1,\a_1+\a_2,r|t)$, with $0 \leq r \leq k$,  $m_1, \dots, m_{k+1-r} \in (\frac{1}{4} +\frac{1}{2} \mathbb{Z})_{<0}$, $m_{k+2-r}, \dots , m_{k+1} \in \mathbb{Z}_{<0}$, and all the other conditions on indices remain the same, and $\widetilde{U(\overline{\n}[\hnu])\overline{\n}[\hnu]_{+}}$ is an appropriate completion of $U(\overline{\n}[\hnu])\overline{\n}[\hnu]_{+}$ (see the Appendix in \cite{S2} for a formal construction of this completion).  Similarly, for any $t \in \Z$ we have 
\be
R(\a_1+\a_2|t)=R^0(\a_1+\a_2|t)+a,\ee
where $a\in U(\overline{\n}[\hnu])\overline{\n}[\hnu]_{+}$ and $R^0(\a_1+\a_2|t)$ is the finite sum indexed by the integers $m_1, \dots, m_{k+1} <0$.

We now note importantly that our relations may be rewritten in a simpler form:
\begin{lem}\label{relationreorder}
We have 
\begin{equation} \label{R-ordered}
\begin{aligned}R^0&(\a_1,\a_1+\a_2,r|t)\\&=\sum_{n=0}^{\lfloor \frac{k+1-r}{2}\rfloor}\sum_{\substack{\mathbf{p}\in (\frac{1}{4}Z)_{<0}^{k+1-n} \\ p_1+\cdots +p_{k+1-n}=-t\\ p_1\leq \cdots \leq p_{k+1-r-2n} \leq -1/4\\ p_{k+2-r-2n} \le \cdots \le p_{k+1-n} \leq -1}}a_{\mathbf{p}}x_{\a_1}^{\hat{\nu}}(p_1)\cdots x_{\a_1}^{\hat{\nu}}(p_{k+1-r-2n})x_{\a_1+\a_2}^{\hat{\nu}}(p_{k+2-r-2n})\cdots x_{\a_1+\a_2}^{\hat{\nu}}(p_{k+1-n})\end{aligned}\end{equation}

where the $a_{\mathbf{p}}\in \mathbb{C}$ are constants related to reordering of the terms .
\end{lem}
\begin{proof}
Consider any monomial in $R^0(\a_1,\a_1+\a_2,r|t)$
\begin{equation} \nonumber
x_{\a_1}^{\hat{\nu}}(m_1+i_1)x_{\a_1}^{\hat{\nu}}(m_2+i_2)\cdots x_{\a_1}^{\hat{\nu}}(m_{k+1-r}+i_{k+1-r})x_{\a_1+\a_2}^{\hat{\nu}}(m_{k+2-r})\cdots x_{\a_1+\a_2}^{\hat{\nu}}(m_{k+1})
\end{equation}
(recall (\ref{R-relations}) and (\ref{R-0-relations})).
Choosing the term with the smallest $m_j+i_j$ for $j=1,\dots,k+1-r$, we move this term to the left of our monomial using the commutation relations (\ref{commutator1}). Namely, we have that:
\begin{equation} \nonumber
x_{\a_1}^{\hat{\nu}}(m_\ell + i_\ell) x_{\a_1}^{\hat{\nu}}(m_j + i_j ) = x_{\a_1}^{\hat{\nu}}(m_j + i_j )x_{\a_1}^{\hat{\nu}}(m_\ell + i_\ell) + c x_{\a_1+\a_2}^{\hat{\nu}}(m_\ell + m_j + i_\ell + i_j)
\end{equation}
for some constant $c \in \mathbb{C}$, which is $0$ in the case that $m_\ell + m_j + i_\ell + i_j \notin \mathbb{Z}$. At each step, this creates a new monomial with either the same number of $x_{\a_1}^{\hat{\nu}}(\cdot)$ terms, or introduces a $x_{\a_1+ \a_2}^{\hat{\nu}}(\cdot)$ term at the expense of two $x_{\a_1}^{\hat{\nu}}(\cdot)$ terms. One repeats the above process by moving the term $x_{\a_1}^{\hat{\nu}}(p)$ with the smallest $p$, to the left of its corresponding monomial. After doing this we will have created monomials which have at most $\lfloor{\frac{k+1-r}{2}}\rfloor$ $x_{\a_1+\a_2}^{\hat{\nu}}(q)$ terms in the monomials. We also use (\ref{commutator2}) to order the terms $x_{\a_1+\a_2}^{\hat{\nu}}(q)$. Now (\ref{R-ordered}) follows.

\end{proof}

\begin{rem}
Of importance in the proof of our upcoming main result, Theorem \ref{presentation}, will be the variance of the entries in the summands of (\ref{R-ordered}). The unique longest term of (\ref{R-ordered}) with the smallest such variance will be called the most ``balanced'' term of the expression. 

\end{rem}

Set
\be
J=\sum_{r=0}^{k}\left(\sum_{t\geq\frac{k+1+3r}{4}}U(\overline{\n}[\hnu])R^{0}(\a_1,\a_1+\a_2,r|t)\right) + \sum_{t \geq k+1}U(\overline{\n}[\hnu])R^{0}(\a_1+\a_2|t).
\ee
Define the ideals
\be
I_{k,0}=J+U(\overline{\n}[\hnu])\overline{\n}[\hnu]_{+}\ee
and
\be
I_{k,i} = I_{k,0} + \sum_{\ell = 0}^{k+1-i} U(\bar{\frak{n}}[\hat{\nu}])x_{\alpha_1}^{\hat{\nu}}\left(-\frac{1}{4}\right)^\ell x_{\alpha_1 + \alpha_2}^{\hat{\nu}}\left(-1\right)^{k+1-i-\ell}
\ee
for $0 \leq i \leq k$.

Recall the virtual subspaces (\ref{virtual_spaces}). There are natural surjective maps
\be
f_{k,i}: U(\bar{\frak{n}}[\hat{\nu}]) \longrightarrow W_{k,i}, \; \; \; a \mapsto a \cdot v_{k,i}.
\ee
The kernel of these maps will be called the {\em presentation} of $W_{k,i}$. Our main goal in this work is to prove that the presentations Ker $f_{k,i}$ are equal to the ideals $I_{k,i}$ above, a result  analogous to results found in \cite{CalLM1}--\cite{CalLM4} and related works.

\section{Shifting Maps}

Continuing to follow \cite{CalLM4}, \cite{CMP}, etc. we define certain maps acting on $U(\overline{\n}[\hnu])$ and $(V_L^T)^{\otimes k}$, which are needed in the proof of the presentation of the virtual subspaces $W_{k,i}$ for $0 \leq i \leq k$.

Let 
\[ \gamma=\frac{1}{2}(\a_1+\a_2) (=\alpha_{1_{(0)}}) \in\h_{(0)},\]
and 
\[ \theta:L\to \mathbb{C}^{\times},\]
defined by 
\[\theta(\a_1)=-i \text{ and } \theta(\a_2)=i,\] 
and extend it  linearly. Consider the Lie algebra homomorphism on $\overline{\mathfrak{n}}[\hnu]$ defined by 
\[ x_{\a}^{\hnu}(m) \mapsto \theta(\a)x_{\a}^{\hnu}(m+\left<\a_{(0)}, \gamma\right>) \]
which extends to an automorphism of $U(\overline{\mathfrak{n}}[\hnu])$, which we will denote by $\tau_{\gamma,\theta}$. In particular, we have 
\be\begin{aligned}
\tau_{\gamma, \theta} (&\xa{1}{m_1}\cdots \xa{1}{m_r}\xb{n_1}\cdots\xb{n_s})\\&=(-i)^r\xa{1}{m_1+\frac{1}{2}}\cdots\xa{1}{m_r+\frac{1}{2}}\xb{n_1+1}\cdots\xb{n_s+1}\end{aligned}\ee
for $m_1, \dots, m_r \in \frac{1}{4} \mathbb{Z}$ and $n_1, \dots, n_s \in \mathbb{Z}$. We also have
\be
\tau^{-1}_{\gamma,\theta}=\tau_{-\gamma,\theta^{-1}}.\ee

\begin{lem} \label{lemma_ideals}
We have 
\be \label{gamma_ideals}
\tau_{\gamma, \theta}(I_{k,k}) = I_{k,0}.
\ee
\end{lem}

\begin{proof}
Notice that for any $0 \leq r \leq k$ and $t \in (1/4) \mathbb{Z}$ we have
\be \label{computations}
\tau_{\gamma, \theta}  \left ( R^0(\alpha_1, \alpha_1+\alpha_2, r | t)  \right)= R^0 \left (\alpha_1, \alpha_1+\alpha_2, r | t- \frac{1}{2} (k+1+r)   \right) + a
\ee
for some $a \in U(\overline{\n}[\hnu])\overline{\n}[\hnu]_{+} $
and for $t \in \mathbb{Z}$ we have 
\be
\tau_{\gamma, \theta}  \left ( R^0(\alpha_1, \alpha_1+\alpha_2 | t)  \right)= R^0(\alpha_1, \alpha_1+\alpha_2 | t- (k+1))+ a,
\ee
for some $a \in U(\overline{\n}[\hnu])\overline{\n}[\hnu]_{+} $, up to nonzero constants due to the character $\theta$. 
It is obvious that
\[
\tau_{\gamma, \theta} \left (U(\overline{\n}[\hnu])\overline{\n}[\hnu]_{+} \right ) \subset U(\overline{\n}[\hnu])\overline{\n}[\hnu]_{+}
\]
and 
\[
x_{\alpha_+\alpha_2}^{\hnu}(0)= \tau_{\gamma, \theta} (x_{\alpha_1+\alpha_2}^{\hnu} (-1)), \; \; \; 
x_{\a_1}^{\hnu} \left ( \frac{1}{4} \right )= \tau_{\gamma, \theta}  \left (x_{\a_1}^{\hnu}  \left ( - \frac{1}{4} \right ) \right )
\]
Now (\ref{gamma_ideals}) follows.
\end{proof}

We recall the map 
\be\begin{aligned}
\psi_{\gamma,\theta}:U(\overline{\mathfrak{n}}[\hnu])&\to U(\overline{\mathfrak{n}}[\hnu])\\
a&\mapsto \tau^{-1}_{\gamma,\theta}(a)\xa{1}{-\frac{1}{4}}\end{aligned}\ee
from \cite{CalLM4}.
Then for any $a \in U(\overline{\mathfrak{n}}[\hnu])$ we have 
\be
\pt(a)=a\xa{1}{-\frac{1}{4}}.
\ee

\begin{lem}We have
\be
\pt(I_{k,k})\subset I_{k, k-1}.\ee
\end{lem}
\begin{proof}

Need to show that $I_{k,k} x_{\a_1}^{\hnu} \left ( - \frac{1}{4} \right ) \subset I_{k,k-1}$. Repeated applications of (\ref{commutator1}) imply that 
\be \label{products}
 \left [ \prod_{j=1}^s \xa{1}{n_j}, \xa{1}{-\frac{1}{4}} \right ]= \sum_{j=1}^s c_j\left(\prod^s_{\substack{r=1\\r\neq j}}\xa{1}{n_r}\right)\xb{n_j-\frac{1}{4}}
\ee
for nonzero constants $c_j$. Then it follows that for $0\leq r\leq k$ we have
\be
 R^0(\a_1,\a_1+\a_2,r|t) \xa{1}{-\frac{1}{4}} =\xa{1}{-\frac{1}{4}}R^0(\a_1,\a_1+\a_2,r|t)+aR^0\left(\a_1,\a_1+\a_2,r+1|t+\frac{1}{4}\right)+b 
\ee
where $a$ is a constant related to the constants $c_j$ in (\ref{products}) and $b\in U(\overline{\mathfrak{n}}[\hnu])\overline{\mathfrak{n}}[\hnu]_{+} $. 
We identify $R^0\left(\a_1,\a_1+\a_2,k+1|t\right)=R^0(\a_1+\a_2|t)$ as appropriate. Further,
\be R^0(\a_1+\a_2|t) \xa{1}{-\frac{1}{4}} =\xa{1}{-\frac{1}{4}}R^0(\a_1+\a_2|t).
\ee
Thus 
\be 
\pt (J) \subset J + U(\overline{\mathfrak{n}}[\hnu]) \overline{\mathfrak{n}}[\hnu]_{+}.
\ee
One can see easily that 

$
\pt \left ( U(\overline{\mathfrak{n}}[\hnu])\overline{\mathfrak{n}}[\hnu]_{+}+ U(\overline{\mathfrak{n}}[\hnu])x_{a_1} ^{\hnu} \left ( - \frac{1}{4} \right ) + U(\overline{\mathfrak{n}}[\hnu]) x_{\a_1+\a_2}^{\hnu}(-1) \right ) \subset U(\overline{\mathfrak{n}}[\hnu])\overline{\mathfrak{n}}[\hnu]_{+}+ U(\overline{\mathfrak{n}}[\hnu])x_{a_1} ^{\hnu} \left ( - \frac{1}{4} \right )^2 + U(\overline{\mathfrak{n}}[\hnu]) x_{\a_1+\a_2}^{\hnu}(-1)x_{a_1} ^{\hnu} \left ( - \frac{1}{4} \right ).
$
Therefore, we have $$\pt(I_{k,k})\subset I_{k, k-1}.$$

\end{proof}

We also have the following result:

\begin{lem} \label{inclusion_ideals}
For all $s,i \ge 0$ with $0 \leq i+s \leq k$ we have:
\begin{equation}
I_{k,k} x_{\alpha_1}^{\hat{\nu}} \left ( - \frac{1}{4} \right )^{k-i-s}\xb{-1}^s \subset I_{k,i}.
\end{equation}

\begin{proof}  We first prove that we have
\begin{equation} \label{i_inclusion}
I_{k,k} x_{\alpha_1}^{\hat{\nu}} \left ( - \frac{1}{4} \right )^{k-i} \subset I_{k,i}
\end{equation}
for all $0 \leq i \leq k$. One can see that
\be \nonumber \begin{aligned}
& R^0(\a_1, \a_1+\a_2, r | t) x_{\a_1}^{\hnu} \left ( - \frac{1}{4} \right )^{k-i} \\ & 
= \sum_{j=0}^{k-i} x_{\a_1}^{\hnu} \left ( - \frac{1}{4} \right )^j R^0 \left (\a_1, \a_1+\a_1+\a_2, r+k-i-j | t+ \frac{(k-i-j)(k-i)}{4} \right) + a,
\end{aligned}
\ee
where $a \in U(\overline{\mathfrak{n}}[\hnu])\overline{\mathfrak{n}}[\hnu]_{+} $, and 
\be \nonumber 
R^0(\a_1, \a_1+\a_2 | t) x_{\a_1}^{\hnu} \left ( - \frac{1}{4} \right )^{k-i} = x_{\a_1}^{\hnu} \left ( - \frac{1}{4} \right )^{k-i}R^0(\a_1, \a_1+\a_2 | t),
\ee
and
\be \nonumber
 R^0(\a_1+\a_2 | t) x_{\a_1}^{\hnu} \left ( - \frac{1}{4} \right )^{k-i} = x_{\a_1}^{\hnu} \left ( - \frac{1}{4} \right )^{k-i}R^0(\a_1+\a_2 | t).
\ee
Then 
\be \nonumber \begin{aligned}
& I_{k,k} x_{\a1}^{\hnu} \left ( -\frac{1}{4} \right )^{k-i} \\ &
\subset J + U(\overline{\mathfrak{n}}[\hnu])\overline{\mathfrak{n}}[\hnu]_{+}  + U(\overline{\mathfrak{n}}) x_{\a_1}^{\hnu}  \left ( -\frac{1}{4} \right ) ^{k+1-i} + U(\overline{\mathfrak{n}}) x_{\a_1+\a_2}^{\hnu} (-1) x_{\a_1}^{\hnu} \left ( -\frac{1}{4} \right )^{k-i} \subset I_{k,i}.
\end{aligned}
\ee

Now let $i, s \geq 0$ such that $0 \leq i+s \leq k$. By (\ref{i_inclusion}) we have 
\begin{equation} \nonumber
I_{k,k} x_{\alpha_1}^{\hat{\nu}}\left ( - \frac{1}{4} \right )^{k-i-s} \subset I_{k,i+s}.
\end{equation}
Hence,
\begin{equation} \nonumber \begin{aligned}
& I_{k,k} x_{\alpha_1}^{\hat{\nu}} \left ( - \frac{1}{4} \right )^{k-i-s}\xb{-1}^s \subset I_{k,i+s}\xb{-1}^s\\ 
& =  I_{k,0}\xb{-1}^s + \sum_{\ell = 0}^{k+1-i-s} U(\bar{\frak{n}}[\hat{\nu}])x_{\alpha_1}^{\hat{\nu}}\left(-\frac{1}{4}\right)^\ell x_{\alpha_1 + \alpha_2}^{\hat{\nu}}\left(-1\right)^{k+1-i-s-\ell}\xb{-1}^s\\ 
&\subset I_{k,0} + \sum_{\ell = 0}^{k+1-i-s} U(\bar{\frak{n}}[\hat{\nu}])x_{\alpha_1}^{\hat{\nu}}\left(-\frac{1}{4}\right)^\ell x_{\alpha_1 + \alpha_2}^{\hat{\nu}}\left(-1\right)^{k+1-i-\ell}\\ 
&\subset I_{k,i}, 
\end{aligned}
\end{equation}
where we use the fact that $\xb{-1}$ commutes with all elements of $I_{k,0}$.
\end{proof}
\end{lem}

Consider the linear map 
\begin{equation} 
e_{\alpha_1}:V_L^T \longrightarrow V_L^T,
\end{equation}
and its restriction to the principal subspace $W_L^T$
\begin{equation} \label{e-map}
e_{\alpha_1}:W_L^T \longrightarrow W_L^T
\end{equation}
Then
\begin{equation}
e_{\alpha_1} (a \cdot 1)=A_{C( \cdot, \cdot )}^{\sigma (\cdot), \theta (\cdot )} \psi_{\gamma, \theta} (a) \cdot 1,
\end{equation}
where $a \in U(\overline{\mathfrak{n}}[\hnu])$ and $A_{C( \cdot, \cdot )}^{\sigma (\cdot), \theta (\cdot )}$ is a nonzero constant depending on the maps $C(\cdot, \cdot)$, $\sigma(\cdot)$ and $\theta(\cdot)$. Now consider the linear map
\[
e_{\alpha_1}^{\otimes k} :(V_L^T)^{\otimes k} \longrightarrow (V_L^T)^{\otimes k}.
\]
Then we have 
\be \label{e_alpha_k}
e_{\alpha_1}^{\otimes k} :W_{k,0} \longrightarrow W_{k,k}
\ee
where
\be \label{e_alpha_map}
e_{\a_1}^{\otimes k} (a \cdot v_{k,0}) =\tau^{-1}_{\gamma, \theta} (a) \cdot v_{k,k},
\ee
up to certain nonzero constants.

Let $\lambda_1$ be the fundamental weight of $\mathfrak{g}$ defined by:
\[
\lambda_1=\frac{2}{3} \a_1+\frac{1}{3} \a_2.
\]
We recall from \cite{CalLM4} the operators 
\be
\Delta^T(\lambda_1, -x) =i^{2 {\lambda_1}_{(0)}} x^{{\lambda_1}_{(0)}} E^{+} (-\lambda_1, x) \in (\mbox{End} \; V_L^T) [[x^{1/4}, x^{-1/4}]],
\ee
where ${\lambda_1}_{(0)}= \frac{1}{2} (\a_1+\a_2)$. Denote by $\Delta_c^T (\lambda_1, -x)$ the constant term of $\Delta^T(\lambda_1, -x)$. Recall also 
\[
\Delta_c^T(\lambda_1, -x) : W_L^T \longrightarrow W_L^T, \; \; \; \; \; 
a \cdot 1 \mapsto \tau_{\gamma, \theta} (a) \cdot 1.
\]

We now define the following operators:
\be \label{Y_operator}
\mathcal{Y}_i = 1_V \otimes \dots \otimes 1_V \otimes e_{\alpha_1} \circ \Delta_c^T(-\lambda_1,x) \otimes 1_V \dots \otimes 1_V : (V_L^T) ^{\otimes k} \longrightarrow (V_L^T)^{\otimes k},
\ee
where the non-identity operator occurs at the $i$-th slot. Using the operators $\mathcal{Y}_{k-i}$ one can show that
\be
\mbox{Ker} \; f_{k, i} \subset \mbox{Ker} \; f_{k, i+1}
\ee
for any $0 \leq i \leq k-1$.

\section{Presentations of the subspaces $W_{k,i}$}
We now have the necessary ingredients to prove a presentation of the spaces $W_{k,i}$.
\begin{thm} \label{presentation}
For $i=0, \dots, k$ we have
\be
\mbox{Ker} \; f_{k,i}= I_{k,i}. 
\ee
\end{thm}

\begin{proof} Let $0 \leq i \leq k$. 
Since the inclusion $I_{k,i} \subset {\mbox Ker} f_{k,i}$ is obvious, the remainder of the proof will show that ${\mbox Ker} f_{k,i} \subset I_{k,i}$ for $i=0,\dots,k$. Suppose that for some $\ell =0,\dots,k$ we have ${\mbox Ker} f_{k,\ell} \not\subset I_{k,\ell}$, and consider the set 
$$
\{a \in U(\overline{\mathfrak{n}}[\hnu]) | a \in Kerf_{k,\ell} \setminus I_{k,\ell}  \text{ for some } \ell =0,\dots ,k \}.
$$
We may and do assume that elements of this set have positive weight. Among the elements of this set, we look at those with lowest charge. Among the elements of lowest charge, we choose an element of lowest weight and call it $a$.

We first show that $\ell \neq k$. Suppose $a \in {\mbox Ker} f_{k,k}\setminus I_{k,k}$. Then, we have that 
\begin{equation}
a\cdot (e_{\alpha_1} \otimes \dots \otimes e_{\alpha_1}) = 0,
\end{equation}
and thus by (\ref{e_alpha_map}) we have
\begin{equation}
e_{\alpha_1}^{\otimes k}(\tau_{\gamma, \theta} (a)(1 \otimes \dots \otimes 1)) = 0. 
\end{equation}
Now the injectivity of $e_{\alpha_1}^{\otimes k}$ implies that 
\begin{equation}
\tau_{\gamma, \theta} (a)(1 \otimes \dots \otimes 1) = 0.
\end{equation}
Since $\tau_{\gamma, \theta} (a)$ has lower weight than $a$, we have that $\tau_{\gamma, \theta} (a) \in I_{k,0}$, and so applying Lemma \ref{lemma_ideals} we have that $a \in I_{k,k}$, a contradiction.

We now proceed by induction and assume that we have shown that $\ell \neq k, k-1, \dots, i+1$ for some $0\leq i \leq k-1$ . We will show that this implies $\ell \neq i$ as well. We suppose our minimal element $a \in {\mbox Ker }f_{k,i}\setminus I_{k,i}$. 

We first show, using an inductive proof, that we can write
\be\label{rewritingfinalgoal}
a=b+c\xb{-1}^{k-i}\ee
with $b\in I_{k,i}$ and $c\in U(\bar{\frak{n}}[\hat{\nu}])$.

Since $a\in {\mbox Ker }f_{k,i}\subset  {\mbox Ker} f_{k,i+1}=I_{k,i+1}$ we may write 
\be\label{rewrite1}
a=b^{(0)}+\sum_{j=0}^{k-i}c_j^{(0)}\xa{1}{-\frac{1}{4}}^{k-i-j}\xb{-1}^j,\ee
where $b^{(0)}\in I_{k,0}$ and $c_j^{(0)}\in U(\bar{\frak{n}}[\hat{\nu}])$ for $0 \leq j \leq k-i$. We will proceed inductively towards (\ref{rewritingfinalgoal}) by first showing that 
\be\label{rewrite2}
c_0^{(0)}\xa{1}{-\frac{1}{4}}^{k-i}\in I_{k,i}+\sum_{\substack{m\geq 0, n\geq 1\\m+n=k-i}}U(\bar{\frak{n}}[\hat{\nu}])\xa{1}{-\frac{1}{4}}^m\xb{-1}^n.\ee
Using the decomposition (\ref{rewrite1}), we have 
\be 
(a-b^{(0)})(1^{\otimes (k-i)}\otimes e_{\a_1}^{\otimes i})=0, 
\ee
and thus
\be
\sum_{j=0}^{k-i}c_j^{(0)}\xa{1}{-\frac{1}{4}}^{k-i-j}\xb{-1}^j(1^{\otimes (k-i)}\otimes e_{\a_1}^{\otimes i})=0\ee
Using the diagonal action we can write 
\be
\sum_{j=0}^{k-i}c_j^{(0)}\left(\left(\xa{1}{-\frac{3}{4}}e_{\a_1}\right)^{\otimes j}\otimes e_{\a_1}^{\otimes (k-j)}+w_{0}\right)=0\ee
where $w_{0}$ is a sum of tensor factors all of which have $\xa{1}{-\frac{3}{4}}e_{\a_1}$ as at least one of the first $k-i$ entries or it is zero. We now write
\be
e_{\a_1}^{\otimes k}\left(\sum_{j=0}^{k-i}\tau_{\gamma,\theta}(c_j^{(0)})\left(e_{\a_1}^{\otimes j}\otimes 1^{\otimes (k-j)}+w^{'}_{0}\right)\right)=0\ee
where $e_{\a_1}^{\otimes k}(w_0^{'})=w_0$ and thus $w_0^{'}$ is a sum of tensor factors all of which have $e_{\a_1}$ as one of the first $k-i$ entries or is zero. By the injectivity of $e_{\a_1}^{\otimes k}$ we have
\be\label{sum1}\sum_{j=0}^{k-i}\tau_{\gamma,\theta}(c_j^{(0)})\left(e_{\a_1}^{\otimes j}\otimes 1^{\otimes (k-j)}+w^{'}_{0}\right)=0\ee
We now apply the composition operators $\mathcal{Y}_1\circ \mathcal{Y}_2 \circ \cdots \circ \mathcal{Y}_{k-i}$ which collapses (\ref{sum1}) to 
 \be
 \tau_{\gamma,\theta}\left(c_{0}^{(0)}\right)\left(e_{\a_1}^{\otimes k-i}\otimes 1^{\otimes i}\right)=0.\ee
So $\tau_{\gamma,\theta}\left(c_{0}^{(0)}\right)\in {\mbox Ker }f_{k,k-i}$. Since $\tau_{\gamma,\theta}\left(c_{0}^{(0)}\right)$ has lower charge than $a$, then
\be\tau_{\gamma,\theta}\left(c_{0}^{(0)}\right)\in I_{k,k-i}=I_{k,0}+\sum_{\substack{m\geq 0,n\geq 0 \\ m+n=i+1}}U(\bar{\frak{n}}[\hat{\nu}])\xa{1}{-\frac{1}{4}}^m\xb{-1}^n. \ee
So we have, using Lemma \ref{lemma_ideals},
\be
c_{0}^{(0)}\in I_{k,k}+\sum_{\substack{m\geq 0,n\geq 0 \\ m+n=i+1}}U(\bar{\frak{n}}[\hat{\nu}])\xa{1}{-\frac{3}{4}}^m\xb{-2}^n\ee
Now putting this into the context of (\ref{rewrite1}) we have 
\be\begin{aligned}
c_{0}^{(0)}\xa{1}{-\frac{1}{4}}^{k-i}\in& I_{k,k}\xa{1}{-\frac{1}{4}}^{k-i}\\&+\sum_{\substack{m\geq 0,n\geq 0 \\ m+n=i+1}}U(\bar{\frak{n}}[\hat{\nu}])\xa{1}{-\frac{3}{4}}^m\xa{1}{-\frac{1}{4}}^{k-i}\xb{-2}^n\end{aligned}\ee
By Lemma \ref{inclusion_ideals} we have $I_{k,k}\xa{1}{-\frac{1}{4}}^{k-i}\subset I_{k,i}$. Now we focus on the terms $\xa{1}{-\frac{3}{4}}^m\xa{1}{-\frac{1}{4}}^{k-i}\xb{-2}^n$ such that $m+n=i+1$. Each of these is a summand (with non-zero coefficient) of the expression $R^0\left(\a_1,\a_1+\a_2,n|\frac{3m+8n+k-i}{4}\right)$ and is the most ``balanced'' term of this expression. That is,  by Lemma \ref{relationreorder},  every other summand of this expression is either an element of $I_{k,i}$ or a $U(\bar{\frak{n}}[\hat{\nu}])$-multiple of 
$$
\xa{1}{-\frac{1}{4}}^{k-i-j}\xb{-1}^j \text{ for } 1\leq j\leq k-i
$$ 
which verifies (\ref{rewrite2}). Using (\ref{rewrite1}) we have
\be
a=b^{(1)}+\sum_{j=1}^{k-i} c_{j}^{(1)}\xa{1}{-\frac{1}{4}}^{k-i-j}\xb{-1}^{j},\ee
for $b_{(1)}\in I_{k,i}$ and $c_{j}^{(1)}\in U(\bar{\frak{n}}[\hat{\nu}])$.

Now, suppose that we have proved
\begin{equation} \label{induction_s}
a = b^{(s)} + \sum_{j=s}^{k-i}c_{j}^{(s)}\xa{1}{-\frac{1}{4}}^{k-i-j}\xb{-1}^j
\end{equation}
for some $b^{(s)} \in I_{k,i}$ and $c_j^{(s)} \in U(\bar{\frak{n}}[\hat{\nu}]),$ where $1\leq s < k-i$. We show that 
\begin{equation}
a = b^{(s+1)} + \sum_{j=s+1}^{k-i}c_{j}^{(s+1)}\xa{1}{-\frac{1}{4}}^{k-i-j}\xb{-1}^j,
\end{equation}
where $b^{(s+1)} \in I_{k,i}$ and $c_j^{(s+1)} \in U(\bar{\frak{n}}[\hat{\nu}])$.

Using the decomposition (\ref{induction_s}), we have 
\be 
(a-b^{(s)})(1^{\otimes (k-i)}\otimes e_{\a_1}^{\otimes i})=0
\ee
and thus
\be
\sum_{j=s}^{k-i}c_j^{(s)}\xa{1}{-\frac{1}{4}}^{k-i-j}\xb{-1}^j(1^{\otimes (k-i)}\otimes e_{\a_1}^{\otimes i})=0\ee
Using the diagonal action we can write 
\be
\sum_{j=s}^{k-i}c_j^{(s)}\left(\left(\xa{1}{-\frac{3}{4}}e_{\a_1}\right)^{\otimes j}\otimes e_{\a_1}^{\otimes (k-j)}+w_{s}\right)=0\ee
where $w_{s}$ is a sum of tensor factors all of which have $\xa{1}{-\frac{3}{4}}e_{\a_1}$ as at least one of the entries between the $s+1$ and $k-i$ position or is zero. We now write
\be
e_{\a_1}^{\otimes k}\left(\sum_{j=s}^{k-i}\tau_{\gamma,\theta}(c_j^{(s)})\left(e_{\a_1}^{\otimes j}\otimes 1^{\otimes (k-j)}+w^{'}_{s}\right)\right)=0\ee
where $e_{\a_1}^{\otimes k}(w_s^{'})=w_s$ and thus $w_s^{'}$ is a sum of tensor factors all of which have $e_{\a_1}$ as at least one of the entries between the $s+1$ and $k-i$ position or is zero. By the injectivity of $e_{\a_1}^{\otimes k}$ we have
\be\label{sum2}\sum_{j=s}^{k-1}\tau_{\gamma,\theta}(c_j^{(s)})\left(e_{\a_1}^{\otimes j}\otimes 1^{\otimes (k-j)}+w^{'}_{s}\right)=0\ee
Now apply the composition of operators $\mathcal{Y}_{s+1}\circ \mathcal{Y}_{s+2}\circ \cdots \circ \mathcal{Y}_{k-i}$ which collapses (\ref{sum2}) to 
\be 
\tau_{\gamma,\theta}\left(c_s^{(s)}\right)\left(e_{\a_1}^{\otimes (k-i)}\otimes 1^{\otimes i}\right)=0.\ee
So $\tau_{\gamma,\theta}\left(c_{s}^{(s)}\right)\in \text{Ker }f_{k,k-i}$ and has lower charge than $a$, thus 
\be\tau_{\gamma,\theta}\left(c_{s}^{(s)}\right)\in I_{k,k-i}=I_{k,0}+\sum_{\substack{m\geq 0,n\geq 0 \\ m+n=i+1}}U(\bar{\frak{n}}[\hat{\nu}])\xa{1}{-\frac{1}{4}}^m\xb{-1}^n.\ee
Using Lemma \ref{lemma_ideals} again, we have 
\be
c_{s}^{(s)}\in I_{k,k}+\sum_{\substack{m\geq 0,n\geq 0 \\ m+n=i+1}}U(\bar{\frak{n}}[\hat{\nu}])\xa{1}{-\frac{3}{4}}^m\xb{-2}^n.\ee
Now putting this into the context of (\ref{induction_s}) we have 
\be\begin{aligned}
c_{s}^{(s)}&\xa{1}{-\frac{1}{4}}^{k-i-s}\xb{-1}^s\in I_{k,k}\xa{1}{-\frac{1}{4}}^{k-i-s}\xb{-1}^s\\&+\sum_{\substack{m\geq 0,n\geq 0 \\ m+n=i+1}}U(\bar{\frak{n}}[\hat{\nu}])\xa{1}{-\frac{3}{4}}^m\xa{1}{-\frac{1}{4}}^{k-i-s}\xb{-2}^n\xb{-1}^s\end{aligned}.\ee
By Lemma \ref{inclusion_ideals}, we have 
$I_{k,k}\xa{1}{-\frac{1}{4}}^{k-i-s}\xb{-1}^s\subset I_{k,i}$
 so we focus on the terms $\xa{1}{-\frac{3}{4}}^m\xa{1}{-\frac{1}{4}}^{k-i-s}\xb{-2}^n\xb{-1}^s$ such that $m+n=i+1$. Each of these is a summand (with non-zero coefficient) of the expression $R^0\left(\a_1,\a_1+\a_2,n+s|\frac{3m+8n+3s+k-i}{4}\right)$ and is the most ``balanced'' term of this expression. That is, by Lemma \ref{relationreorder}, every other summand of this expression is either an element of $I_{k,i}$ or a $U(\bar{\frak{n}}[\hat{\nu}])$-multiple of 
$$
\xa{1}{-\frac{1}{4}}^{k-i-j}\xb{-1}^j \text{ for } s+1\leq j\leq k-i,
$$ 
which finishes the inductive argument that verifies the claim described by (\ref{rewritingfinalgoal}). Then we write
\be a=b+c\xb{-1}^{k-i}\ee
with $b\in I_{k,i}$ and $c\in U(\bar{\frak{n}}[\hat{\nu}])$. Now since 
\be
(a-b)\left(1^{\otimes (k-i)}\otimes e_{\a_1}^{\otimes i}\right)=0\ee 
we have 
\be
c\xb{-1}^{k-i}\left(1^{\otimes (k-i)}\otimes e_{\a_1}^{\otimes i}\right)=0\ee
and thus 
\be
c\left(\left(\xa{1}{-\frac{3}{4}}e_{\a_1}\right)^{\otimes (k-i)}\otimes e_{\a_1}^{\otimes i}\right)=0\ee
which implies 
\be e^{\otimes k}_{\a_1}\left(\tau_{\gamma,\theta}(c)(e_{\a_1}^{\otimes (k-i)}\otimes 1^{\otimes i})\right)=0\ee
and by the injectivity of $e^{\otimes k}_{\a_1}$ we have 
\be \tau_{\gamma,\theta}(c)(e_{\a_1}^{\otimes (k-i)}\otimes 1^{\otimes i})\ee
so $\tau_{\gamma,\theta}(c)\in \text{Ker }f_{k,k-i}$. Since  $\tau_{\gamma,\theta}(c)$ has lower charge that $a$, then $\tau_{\gamma,\theta}(c)\in I_{k,k-i}$. Then 
\be
c\in I_{k,k}+\sum_{\substack{m\geq 0, n\geq 0\\m+n=i+1}}U(\bar{\frak{n}}[\hat{\nu}])\xa{1}{-\frac{3}{4}}^m\xb{-2}^n,\ee
and so  
\be
c\xb{-1}^{k-i}\in I_{k,k}\xb{-1}^{k-i}+\sum_{\substack{m\geq 0, n\geq 0\\m+n=i+1}}U(\bar{\frak{n}}[\hat{\nu}])\xa{1}{-\frac{3}{4}}^m\xb{-2}^n\xb{-1}^{k-i}.\ee

By Lemma \ref{inclusion_ideals} we have $I_{k,k}\xb{-1}^{k-i}\subset I_{k,i}$, so we focus on the terms $\xa{1}{-\frac{3}{4}}^m\xb{-2}^n\xb{-1}^{k-i}$ for $m+n=i+1$.  Each of these is a summand (with non-zero coefficient) of the expression $R^0\left(\a_1,\a_1+\a_2,n+k-i|\frac{3m+8n+4k-4i}{4}\right)$ and is the most ``balanced'' term of this expression. That is, every other summand of this expression is and element of $I_{k,i}$. So we have $c\xb{-1}^{k-i}\in I_{k,i}$ and thus $a\in I_{k,i}$, a contradiction which finishes the proof.

\end{proof}

\section{Some remarks on short exact sequences and characters}

In this section we construct a set of short exact sequences among certain virtual subspaces. As a consequence we obtain a system of recursions (q-difference equations) for the characters of virtual subspaces. We also define generalizations of the virtual subspaces and conjecture a more general set of exact sequences and recursions. At the end of this section we give a conjecture for a specialized character of the principal subspace $W_{k,0}$ as a Nahm sum of the inverse of the tadpole Cartan matrix. In particular, when $k$ is even, this character is related to a Gollnitz-Gordon-Andrews identity. 

The following proposition shows that the ideal $I_{k,i}$ may be rewritten to look similar to the ideals found in \cite{CLM2} and \cite{CalLM2}.

\begin{prop}\label{ideal_rewrite}
For $m_1,m_2\geq 0$ such that $m_1+m_2\leq k$, we have
\be\label{rewrite1} U(\overline{\mathfrak{n}}[\hat{\nu}])x_{\alpha_1}^{\hat{\nu}}\left(-\frac{1}{4}\right)^{m_1}\xb{-1}^{m_2}\subset J+U(\overline{\mathfrak{n}}[\hat{\nu}])x_{\alpha_1}^{\hat{\nu}}\left(-\frac{1}{4}\right)^{m_1+m_2}\ee
and thus
\be\label{rewrite2}
I_{k,i}=J+U(\overline{\mathfrak{n}}[\hat{\nu}])\xa{1}{-\frac{1}{4}}^{k+1-i}.\ee

\end{prop}

\begin{proof}
 We will focus on establishing (\ref{rewrite1}) with $m_1=k+1$ and $m_2=1$ as the general case will follow similarly via induction. This will follow from the claim that we may write 
 \be 
 R\left(\a_1,\a_1+\a_2,0\left|\frac{k+3}{4}\right.\right)=A\xa{1}{-\frac{3}{4}}\xa{1}{-\frac{1}{4}}^k+B\xa{1}{-\frac{1}{4}}^{k-1}\xb{-1},
 \ee
where $A,B\neq 0$. It is clear that we may write this expression in this form with $A\neq 0$ using the commutations relations so we will focus on proving that $B\neq 0$. Observe that before any reordering we may write 
\be 
 R\left(\a_1,\a_1+\a_2,0\left|\frac{k+3}{4}\right.\right)=\sum_{m_1+\cdots+m_{k+1}=-\frac{k+3}{4}}A_{\mathbf{m}}\xa{1}{m_1}\cdots\xa{1}{m_{k+1}}\ee
We now analyze which choices of $\mathbf{m}=(m_1,\dots,m_{k+1})$ will allow the monomial $\xa{1}{m_1}\cdots\xa{1}{m_{k+1}}$ to commute and produce a $\xa{1}{-\frac{1}{4}}^{k-1}\xb{-1}$ term. All such tuples will be permutations of $(-\frac{3}{4}-\frac{a}{2},-\frac{1}{4},\dots,-\frac{1}{4},-\frac{1}{4}+\frac{a}{2})$ for $a\geq 0$.

We focus first on the case when $a=0$ that is, permutations of $\left(-\frac{3}{4},-\frac{1}{4},\dots,-\frac{1}{4}\right)$ of which there are exactly $k+1$ depending on our choice of placement of $-\frac{3}{4}$. Notice that we have 
\be \nonumber
\xa{1}{-\frac{1}{4}}^{n_1}\xa{1}{-\frac{3}{4}}\xa{1}{-\frac{1}{4}}^{n_2}=\xa{1}{-\frac{3}{4}}\xa{1}{-\frac{1}{4}}^{k}+\frac{n_1}{2} \xa{1}{-\frac{1}{4}}^{k-1}\xb{-1}\ee
and so that rewriting these terms will account the inclusion of $\frac{k(k+1)}{4}\xa{1}{-\frac{1}{4}}\xb{-1}$.

We now move on to the case when $a\geq 1$ and notice that for all such permutations that are such that $-\frac{3}{4}-\frac{a}{2}$ is to the right of $-\frac{1}{4}+\frac{a}{2}$, applying the commutation relation will produce exactly one $\frac{(-1)^a}{2}\xb{-1}$ term. We can easily count these tuples as follows. Given the general form 
\be \nonumber
\left(\mathbf{m}^{'},-\frac{3}{4}-\frac{a}{2},-\frac{1}{4},\dots,-\frac{1}{4}\right)\ee
where there are $r$ trailing $-\frac{1}{4}$-terms and $\mathbf{m}^{'}$ is a $k-r$ tuple that is a permutation of $(-\frac{1}{4}+\frac{a}{2},-\frac{1}{4},\dots,-\frac{1}{4})$, of which there are exactly $k-r$. This amounts to a total of $\frac{k(k+1)}{2}$ permutations of  $(-\frac{3}{4}-\frac{a}{2},-\frac{1}{4},\dots,-\frac{1}{4},-\frac{1}{4}+\frac{a}{2})$ for $a\geq 1$ which gives us an additional $(-1)^a\frac{k(k+1)}{4}\xb{-1}$ terms upon applying commutation relations. 

Putting this all together we may write 
 \be\begin{aligned} 
 R\left(\a_1,\a_1+\a_2,0\left|\frac{k+3}{4}\right.\right)&=A\xa{1}{-\frac{3}{4}}\xa{1}{-\frac{1}{4}}^k\\&+\frac{k(k+1)}{4}\sum_{a\geq 0}(-1)^aB_a\xa{1}{-\frac{1}{4}}^{k-1}\xb{-1},
\end{aligned} \ee
where $B_a$ is the number of inner sums of the expression in its initial form where a permutation of $(-\frac{3}{4}-\frac{a}{2},-\frac{1}{4},\dots,-\frac{1}{4},-\frac{1}{4}+\frac{a}{2})$ occurs. It is clear that this sum will be finite and that the $B_a$ are decreasing, finishing our proof.

\end{proof}

Recall the maps (\ref{e_alpha_k}) and (\ref{Y_operator}). By using the previous proposition we can prove:

 \begin{thm}\label{exact}The following sequences are exact:
\begin{equation} \label{sequence1}
0\to W_{k,0} \xrightarrow{e_{\alpha_1}^{\otimes k}}W_{k,0}\xrightarrow{\mathcal{Y}_1}W_{k,1}\to 0
\ee
and
\begin{equation} \label{sequence2}
0\to W_{k,0} \xrightarrow{e_{\alpha_1}^{\otimes k}}W_{k,k}\to 0.
\ee
\end{thm}
\begin{proof}
First, it is clear that $e_{\alpha_1}^{\otimes k}$ is an injection and that $\mathcal{Y}_1$ is a surjection. 
It is easy to see that Im$e_{\alpha_1}^{\otimes k} \subset $ Ker $\mathcal{Y}_1$, since for any $a \in  U(\overline{\mathfrak{n}}[\hat{\nu}])$ we have
\begin{align}
\mathcal{Y}_1(e_{\alpha_1}^{\otimes k}(a \cdot v_{k,0})) &= \mathcal{Y}_1(\tau^{-1}_{\gamma,\theta}(a) \xa{1}{-\frac{1}{4}}^k \cdot v_{k,0})\\
&=\tau^{-1}_{\gamma,\theta}(a) \xa{1}{-\frac{1}{4}}^k \cdot v_{k,1}\\
&= 0.
\end{align}

Now, suppose $a \cdot v_{k,0} \in \mbox{Ker } \mathcal{Y}_1$ for some $a \in  U(\overline{\mathfrak{n}}[\hat{\nu}])$. Then we have that 
\be
a \cdot v_{k,1} = 0,
\ee
and so 
\be
a \in \mbox{Ker }f_{k,1} = I_{k,1} =J+U(\overline{\mathfrak{n}}[\hat{\nu}])\xa{1}{-\frac{1}{4}}^{k}
\ee
by Proposition \ref{ideal_rewrite}. So, we may write
\be
a = b + c \xa{1}{-\frac{1}{4}}^k
\ee
for some $b \in J$ and $c \in U(\overline{\mathfrak{n}}[\hat{\nu}])$. We have
\be
a \cdot v_{k,0} =  \left (b + c \xa{1}{-\frac{1}{4}}^k \right ) \cdot v_{k,0} = c \xa{1}{-\frac{1}{4}}^k \cdot v_{k,0} =  e_{\alpha_1}^{\otimes k} (\tau_{\gamma, \theta}(c)\cdot v_{k,0}),
\ee
and thus, $a \cdot v_{k,0} \in \mbox{Im } e_{\alpha_1}^{\otimes k}$. Hence $\mbox{Ker }\mathcal{Y}_1 \subset \mbox{Im } e_{\alpha_1}^{\otimes k}$, completing the proof.

The exactness of the second sequence follows by using Lemma \ref{lemma_ideals}.
\end{proof}

We define the character (or multigraded dimension) of $W_{k,i}$ by
\begin{equation}
\chi_{W_{k,i}}(x,q) = {\rm tr}|_{W_{k,i}}x^{\alpha_1+\alpha_2}q^{L^{\hat{\nu}}(0)}
\end{equation}
From these exact sequences we have the following recursions on the characters of the virtual subspaces:
\begin{cor}
\be \label{recursion1}
\chi_{W_{k,0}}'(x,q) = \chi_{W_{k,1}}'(x,q) + x^k q^{\frac{k}{4}} \chi_{W_{k,0}}'(xq^{\frac{1}{2}},q)
\ee
and
\be\label{0krecursion}
\chi_{W_{k,k}}'(x,q) = \chi_{W_{k,0}}'(xq^{\frac{1}{2}},q)
\ee
where $\chi_{W_{k,i}}'(x,q) = \sum_{r,s \in \frac{1}{4}\mathbb{Z}} \mbox{dim}({W_{k,i}}_{(r,s)})x^rq^s$ is appropriately shifted to be contained in $\mathbb{C}[[x,q^{\frac{1}{4}}]]$.
\end{cor}
\begin{proof}
The first equation follows immediately from the short exact sequence (\ref{sequence1}) that can be  written as:
\begin{equation}
0\to {W_{k,0}}_{(r-k,s-\frac{1}{2}r + \frac{1}{4} k)} \xrightarrow{e_{\alpha_1}^{\otimes k}}{W_{k,0}}_{(r,s)}\xrightarrow{\mathcal{Y}_1}{W_{k,1}}_{(r,s)}\to 0,
\ee
and from which we have
\be
\sum _{r, s \in \frac{1}{4} \mathbb{Z}} \mbox{dim}({W_{k,0}}_{(r,s)})x^rq^s =\sum_{r, s \in \frac{1}{4} \mathbb{Z}} \mbox{dim}({W_{k,1}}_{(r,s)}) x^rq^s+\sum _{r, s \in \frac{1}{4} \mathbb{Z}} \mbox{dim}({W_{k,0}}_{(r-k,s-\frac{1}{2}r + \frac{1}{4}k)}) x^rq^s.
\ee
The second equation in the claim follows similarly by using (\ref{sequence2}). 
\end{proof}

\begin{rem}
When setting $k=1$, the above results recover the exact sequences and recursions found in [CalLM4], where the character of $W_{1,0}$ was defined using $4L^{\nu}(0)$ to measure the conformal weights. When $k > 1$, these results do not yield a full set of recursions or the full characters of the principal subspace or virtual subspaces.
\end{rem}

In general, we conjecture that a full set of recursions will require a more general set of virtual subspace or longer exact sequences involving the other level $k$ standard modules for $A_2^{(2)}$. We now define generalizations of the virtual subspaces $W_{k,i}$ as follows:
\[
W_{k,i, j}^{\alpha_1, \alpha_1+\alpha_2}=U ( \overline{\mathfrak{n}}[\hat{\nu}])  \cdot v_{k,i,j}, \]
where 
\[
v_{k,i,j}= \underbrace{1\ \otimes \cdots \otimes 1}_{k-i-j \; \; \mbox{times}} \otimes \underbrace{ e_{\alpha_1} \otimes \cdots \otimes e_{\alpha_1}}_{i \; \; \mbox{times}} \otimes \underbrace{ x_{\alpha_1+\alpha_2}(-1) \cdot 1 \otimes x_{\alpha_1+\alpha_2}(-1) \cdot 1}_{j \; \; \mbox{times}},
\]
and $0 \leq i, j \leq k$ such that $i+j\le k$. When $j=0$ these are the spaces $W_{k,i}$. 

We conjecture the following: 
\begin{conj}
For any $i$ with $0 \leq i \leq k-1$, the following sequences are exact:
\end{conj}
\begin{equation}
0\to W_{k,i,0}^{\alpha_1, \alpha_1+\alpha_2} \xrightarrow{e_{\alpha_1}^{\otimes k}}W_{k,0, i}^{\alpha_1, \alpha_1+\alpha_2}\xrightarrow{\mathcal{Y}_{i+1}}W_{k,1, i}^{\alpha_1, \alpha_1+\alpha_2} \to 0.
\ee
We also have the following system of recursions satisfied by their characters:
\be
\chi'_{W_{k,0, i}^{\alpha_1, \alpha_1+\alpha_2}} (x,q)= \chi'_{W_{k,1, i}^{\alpha_1, \alpha_1+\alpha_2}}(x,q)+ x^{k-i}q^{\frac{1}{4} (k-i)} \chi'_{W_{k,i,0}^{\alpha_1, \alpha_1+\alpha_2}}(x,q).
\ee
If $i=0$ we recover (\ref{sequence1}) and (\ref{recursion1}).  


Using experimental evidence to compute basis elements of $W_{k,0}$ from the presentation of $W_{k,0}$, we have the following conjecture:
\begin{conj}
For each integer $k > 0$, we have that:
\begin{equation}
\chi_{W_{k,0}}'(1,q^4) = \sum_{r_1,\dots,r_k \ge 0} \frac{q^{B^{st}r_sr_t}}{(q^2;q^2)_{r_1} \dots (q^2;q^2)_{r_k}}
\end{equation}
where $B^{st} = min \{ s,t\}$ for $1 \le s,t \le k$. Thus,
\begin{equation} \label{character}
\chi_{W_{k,0}}'(1,q^4) = \sum_{r_1,\dots,r_k \ge 0} \frac{q^{{\bf r}^tT_k^{-1}{\bf r}}}{(q^2;q^2)_{r_1} \dots (q^2;q^2)_{r_k}},
\end{equation}
where $T_k$ is the tadpole Cartan matrix of rank k, ${\bf r}=(r_1, \dots, r_k)$ and $(q^2;q^2)_n=\prod_{j=0}^{n-1} (1-q^{2+2j})$.

In particular, if $k=2m$ is even, $m$ a positive integer, we have that $\chi_{W_{k,0}}'(1,q^4)$ is the generating function for the following Gollnitz-Gordon-Andrews identity:
\begin{equation}
\chi_{W_{k,0}}'(1,q^4) =  \prod_{\substack{n\geq 1,\,n\not\equiv 2\,(\mathrm{mod}\,4), \\ n\not\equiv 0,\pm(2m+1)\,(\mathrm{mod}\,4m+4)}} (1-q^n)^{-1}.
\end{equation}
\end{conj}

The formula  (\ref{character}) is the analogue of the formula for the character of $W_{k,0}$ in the  $A_1^{(1)}$ case, which is given by the Nahm sum of the inverse of the Cartan matrix of $A_k$ (\cite{CLM2}).

\vspace{.4in}

\noindent{\small \sc Department of Mathematics, The Graduate Center and New York City College
  of Technology, City University of New York, New York, NY 10016}\\
{\em E--mail address}: ccalinescu@citytech.cuny.edu

\vspace{.3in}

\noindent{\small \sc Department of Mathematics, Randolph College, Lynchburg, VA 24503}\\
{\em E--mail address:} mpenn@randolphcollege.edu

\vspace{.3in}

\noindent{\small \sc Department of Mathematics and Computer Science, Ursinus College, Collegeville, PA 19426}\\
{\em E--mail address:} csadowski@ursinus.edu


\begin{thebibliography}{CalLM4}




\bibitem[B]{B} R. E. Borcherds, Vertex algebras, Kac-Moody algebras,
 and the Monster, {\em Proc. Natl. Acad. Sci. USA} {\bf 83} (1986),
 3068--3071.

\bibitem[Bu1]{Bu1} M. Butorac, Combinatorial bases of principal subspaces for the affine Lie algebra of type $B_2^{(1)}$, J. Pure Appl.  Algebra {\bf 218} , (2014), 424-447.

\bibitem[Bu2]{Bu2} M. Butorac, Quasi-particle bases of principal subspaces for the affine Lie algebras of types $B_l^{(1)}$ and $C_l^{(1)}$, {\em Glas. Mat. Ser. } III 51 (2016), 59-108.

\bibitem[Bu3]{Bu3} M. Butorac, Quasi-particle bases of principal subspaces of the affine Lie algebra of type $G_2^{(1)}$, {\em Glas. Mat. Ser.} III 52 (2017), 79-98.

\bibitem[BS]{BS} M. Butorac and C. Sadowski, Combinatorial bases of principal subspaces of modules for twisted affine Lie algebras of type $A_{2l-1}^{(2)}$, $D_l^{(2)}$, $E_6^{(2)}$ and $D_4^{(3)}$, preprint.


\bibitem[C]{C} C. Calinescu, Principal subspaces of higher-level standard $\widehat{{\mathfrak sl}(3)}$-modules, {\em J. Pure Appl. Algebra} {\bf 2} (2010), 2007, 559-575.

\bibitem[CalLM1]{CalLM1} C. Calinescu, J. Lepowsky and A. Milas,
  Vertex-algebraic structure of the principal subspaces of certain
  $A_{1}^{(1)}$-modules, I: level one case, {\em Internat. J. Math.}
  {\bf 19} (2008), 71--92.

\bibitem[CalLM2]{CalLM2} C. Calinescu, J. Lepowsky and A. Milas,
  Vertex-algebraic structure of the principal subspaces of certain
  $A_1^{(1)}$-modules, II: higher level case, {\em J. Pure
    Appl. Algebra} {\bf 212} (2008), 1928--1950.

\bibitem[CalLM3]{CalLM3} C. Calinescu, J. Lepowsky and A. Milas,
  Vertex-algebraic structure of the principal subspaces of level one
  modules for the untwisted affine Lie algebras of types A, D, E, {\em
    J. Algebra} {\bf 323} (2010), 167--192.

\bibitem[CalLM4]{CalLM4} C. Calinescu, J. Lepowsky and A. Milas, 
Vertex-algebraic structure of principal subspaces of standard 
$A_2^{(2)}$-modules, I, {\em Internat. J. Math.} {\bf 25} (2014).

\bibitem[CLM1]{CLM1} S. Capparelli, J. Lepowsky and A. Milas, The
Rogers-Ramanujan recursion and intertwining operators, {\em Comm. in
Contemp. Math.} {\bf 5} (2003), 947--966.

\bibitem[CLM2]{CLM2} S. Capparelli, J. Lepowsky and A. Milas, The
Rogers-Selberg recursions, the Gordon-Andrews identities and
intertwining operators, {\em The Ramanujan Journal} {\bf 12} (2006),
379--397.

\bibitem[CMP]{CMP} C. Calinescu, A. Milas and M. Penn, Vertex-algebraic structure of principal subspaces of basic $A_{2n}^{(2)}$-modules {\em J. Pure Appl. Algebra} {\bf 220} (2016), 1752-1784.




    
 
 
\bibitem[FS1]{FS1} B. Feigin and A. Stoyanovsky, Quasi-particles
  models for the representations of Lie algebras and geometry of flag
  manifold; arXiv:hep-th/9308079.

\bibitem[FS2]{FS2} B. Feigin and A. Stoyanovsky, Functional models for
  representations of current algebras and semi-infinite Schubert cells
  (Russian), {\em Funktsional Anal. i Prilozhen.} {\bf 28} (1994),
  68--90; translation in: {\em Funct. Anal. Appl.} {\textbf 28}
  (1994), 55--72.

\bibitem[FHL]{FHL} I. Frenkel, Y.-Z. Huang and J. Lepowsky, On axiomatic
  approaches to vertex operator algebras and modules, {\it Memoirs
    American Math. Soc.} {\bf 104}, 1993.



\bibitem[FLM1]{FLM1} I. Frenkel, J. Lepowsky and A. Meurman, Vertex
  operator calculus, in: {\em Mathematical Aspects of String Theory,
    Proc. 1986 Conference, San Diego,} ed. by S.-T. Yau, World
  Scientific, Singapore, 1987, 150--188.

\bibitem[FLM2]{FLM2} I. Frenkel, J. Lepowsky and A. Meurman, {\em
    Vertex Operator Algebras and the Monster}, Pure and Applied Math.,
  Vol. 134, Academic Press, 1988.

\bibitem[Je]{Je} M. Jerkovic, Recurrences and characters of
Feigin-Stoyanovsky's type subspaces. Vertex operator algebras and
related areas, {\em Contemp. Math.} {\bf 497} (2009), 113--123.

\bibitem[K]{K} V. Kac, {\em Infinite Dimensional Lie Algebras}, $3$rd
  edition, Cambridge University Press, 1990.
  
\bibitem[Ko]{Ko}
S. Ko\v{z}i\'{c}, 
 Principal subspaces for quantum affine algebra $U_{q}(A_{n}^{(1)})$, 
{\em J. Pure Appl. Algebra} \textbf{218} (2014), 2119--2148.



\bibitem[L1]{L1} J. Lepowsky, Calculus of twisted vertex operators,
  {\em Proc.  Nat. Acad. Sci. USA} {\bf 82} (1985), 8295--8299.

\bibitem[L2]{L2} J. Lepowsky, Perspectives on vertex operators and the
  Monster, in: Proc. 1987 Symposium on the Mathematical Heritage of
  Hermann Weyl, Duke Univ., {\em Proc. Symp. Pure Math., American
    Math. Soc.} {\bf 48} (1988), 181--197.

\bibitem[LL]{LL} J. Lepowsky and H. Li, {\it Introduction to Vertex
    Operator Algebras and Their Representations}, Progress in
  Mathematics, Vol. 227, Birkh\"auser, Boston, 2003.

\bibitem[Li]{Li} H.-S. Li, Local systems of twisted vertex
  operators, vertex superalgebras and twisted modules, {\em
    Contemp. Math.} {\bf 193} (1996), 203--236.


\bibitem[MP]{MP} A. Milas and M. Penn, Lattice vertex algebras and
  combinatorial bases: general case and $\mathcal{W}$-algebras, {\em
    New York J. Math.} {\bf 18} (2012), 621--650.


\bibitem[P]{P} M. Penn, Lattice Vertex Superalgebras I: Presentation of the Principal Subspace, {\em Communications in Algebra.}{Volume 42, Issue 3} (2014), 933-961


\bibitem[PS1]{PS1} M. Penn and C. Sadowski, Vertex-algebraic structure of principal subspaces of basic $D_4^{(3)}$-modules, {\em The Ramanujan Journal}, 43:4 (2017), 571-617.

\bibitem[PS2]{PS2} M. Penn and C. Sadowski, 
Vertex-algebraic structure of principal subspaces of basic modules for twisted affine Kac-Moody Lie algebras of type $A_{2n+1}^{(2)}, D_n^{(2)}, E_6^{(2)}$,  \emph{J. Algebra} {Volume 496}, (2018), pp.  242-291 

\bibitem[PSW]{PSW} M. Penn, C. Sadowski, and G. Webb
 \emph{Twisted Modules of Principal Subalgebras of Lattice Vertex Algebras}, 40 pgs, \textit{submitted}

\bibitem[Pr]{Pr} M. Primc, $(k,r)$-admissible configurations and intertwining operators, {\em Contemp. Math} {\bf 422} (2007), 425-434.


\bibitem[S1]{S1} C. Sadowski, Presentations of the principal suspaces of higher level $\widehat{\mathfrak{sl}(3)}$-modules, {\em J. Pure Appl. Algebra}, {\bf 219} (2015), 2300-2345.

\bibitem[S2]{S2} C. Sadowski, Principal subspaces of $\widehat{\mathfrak{sl}(n)}$-modules, {\em Int. J. Math.}, Vol. 26 No. 08, 1550063 (2015).


\bibitem[T1]{T1} G. Trupcevic, Combinatorial bases of Feigin-Stoyanovsky's type subspaces of level one standard modules for $\tilde{\mathfrak{sl}}(l+1, \mathbb{C})$, {\em Comm. Algebra} {\bf 38} (2010), 3913-3940.

\bibitem[T2]{T2} G. Trupcevic, Combinatorial bases of Feigin-Stoyanovsky's type subspaces of higher-level standard $\tilde{\mathfrak{sl}}(l+1, \mathbb{C})$-modules, {\em J. Algebra} {\bf 322} (2009), 3744-3774.






\end{thebibliography}
\end{document}